\documentclass[12pt]{article}
\usepackage{amsmath,amssymb,amsthm,comment}

\newtheorem{theorem}{Theorem}[section]
\newtheorem{thmy}{Theorem}

\newtheorem{lemma}[theorem]{Lemma}

\def\barr{\begin{array}}
\def\earr{\end{array}}

\title{Finite groups with dense solitary subgroups}
\author{Marius T\u arn\u auceanu}
\date{December 11, 2024}

\begin{document}

\maketitle

\begin{abstract}
A group $G$ is said to have dense solitary subgroups if each non-empty open interval in its subgroup lattice $L(G)$ contains a solitary subgroup. In this short note, we find all finite groups satisfying this property.
\end{abstract}

{\small
\noindent
{\bf MSC2000\,:} Primary 20D60; Secondary 20D30, 20E99.

\noindent
{\bf Key words\,:} finite groups, ZM-groups, solitary subgroups, density.}

\section{Introduction}

Let $\mathcal{X}$ be a property pertaining to subgroups of a group. We say that a group $G$ has \textit{dense $\mathcal{X}$-subgroups} if for each pair $(H,K)$ of subgroups of $G$ such that $H<K$ and $H$ is not maximal in $K$, there exists a $\mathcal{X}$-subgroup $X$ of $G$ such that $H<X<K$. Groups with dense $\mathcal{X}$-subgroups have been studied for many properties $\mathcal{X}$: being a normal subgroup \cite{10}, a pronormal subgroup \cite{15}, a normal-by-finite subgroup \cite{5}, a nearly normal subgroup \cite{6} or a subnormal subgroup \cite{7}.

In our note, we will focus on the property of being a solitary subgroup. We recall that a subgroup $H$ of a group $G$ is called a \textit{solitary subgroup} of $G$ if $G$ does not contain another isomorphic copy of $H$. By Theorem 25 of \cite{9}, the set ${\rm Sol}(G)$ of solitary subgroups of $G$ forms a lattice with respect to set inclusion, which is called the \textit{lattice of solitary subgroups} of $G$. This has been investigated in several recent paper, such as \cite{1,4,13}.

We also recall that a \textit{ZM-group} is a finite non-abelian group with all Sylow subgroups cyclic. By \cite{8}, such a
group is of type
\begin{equation}
{\rm ZM}(m,n,r)=\langle a, b \mid a^m = b^n = 1,\, b^{-1} a b = a^r\rangle,\nonumber
\end{equation}where the triple $(m,n,r)$ satisfies the conditions
\begin{equation}
{\rm gcd}(m,n) = {\rm gcd}(m, r-1) = 1 \mbox{ and } r^n \equiv 1 \hspace{1mm}({\rm mod}\hspace{1mm}m).\nonumber
\end{equation}It is clear that $|{\rm ZM}(m,n,r)|=mn$ and $Z({\rm ZM}(m,n,r))=\langle b^d\rangle$, where $d$ is the multiplicative order of
$r$ modulo $m$, i.e.
\begin{equation}
d=o_m(r)={\rm min}\{k\in\mathbb{N}^* \mid r^k\equiv 1 \hspace{1mm}({\rm mod} \hspace{1mm}m)\}.\nonumber
\end{equation}The subgroups of ${\rm ZM}(m,n,r)$ have been completely described
in \cite{3}. Set
\begin{equation}
L=\left\{(m_1,n_1,s)\in\mathbb{N}^3 \,:\, m_1|m, n_1|n,s<m_1, m_1|s\frac{r^n-1}{r^{n_1}-1}\right\}.\nonumber
\end{equation}Then there is a bijection between $L$ and the subgroup lattice
$L({\rm ZM}(m,n,r))$ of ${\rm ZM}(m,n,r)$, namely the function
that maps a triple $(m_1,n_1,s)\in L$ into the subgroup $H_{(m_1,n_1,s)}$ defined by
\begin{equation}
H_{(m_1,n_1,s)}=\bigcup_{k=1}^{\frac{n}{n_1}}\alpha(n_1,s)^k\langle a^{m_1}\rangle=\langle a^{m_1},\alpha(n_1, s)\rangle,\nonumber
\end{equation}where $\alpha(x, y)=b^xa^y$, for all $0\leq x<n$ and $0\leq y<m$. Remark that $|H_{(m_1,n_1,s)}|=\frac{mn}{m_1n_1}$\,, for any $s$ satisfying $(m_1,n_1,s)\in L$.
\bigskip

Our main result is stated as follows.

\begin{theorem}
A finite group $G$ has dense solitary subgroups if and only if $G$ is either cyclic or a ZM-group ${\rm ZM}(m,n,r)$, where $m$ is prime, $d=o_m(r)$ is prime and $n=d^{\alpha}p^{\beta}$ with $p$ prime, $p\neq d$, $\alpha\geq 1$ and $\beta\in\{0,1\}$.
\end{theorem}

Note that the symmetric group $S_3={\rm ZM}(3,2,2)={\rm SmallGroup}(6,1)$ and ${\rm ZM}(13,6,2)={\rm SmallGroup}(78,4)$ are examples of such groups with $d=2$, $\alpha=1$, $p=3$ and $\beta=0$ or $\beta=1$, respectively. Also, the dicyclic group ${\rm Dic}_3={\rm ZM}(3,4,2)={\rm SmallGroup}(12,1)$ and ${\rm ZM}(13,12,12)={\rm SmallGroup}(156,4)$ are examples of such groups with $d=\alpha=2$, $p=3$ and $\beta=0$ or $\beta=1$, respectively.

For the proof of Theorem 1.1, we need the following well-known result (see e.g. (4.4) of \cite{12}, II).

\begin{thmy}
A finite $p$-group has a unique subgroup of order $p$ if and only if it is either cyclic or a generalized quaternion $2$-group.
\end{thmy}

By a \textit{generalized quaternion $2$-group} we will understand a group of order $2^n$ for some positive integer $n\geq 3$, defined by
\begin{equation}
Q_{2^n}=\langle a,b \mid a^{2^{n-2}}= b^2, a^{2^{n-1}}=1, b^{-1}ab=a^{-1}\rangle.\nonumber
\end{equation}
 
Most of our notation is standard and will not be repeated here. Basic definitions and results on groups can be found in \cite{8,12}. For subgroup lattice concepts we refer the reader to \cite{11}.

\section{Proofs of the main results}

We start with two auxiliary results. The first one shows that cyclic groups are the unique $p$-groups with dense solitary subgroups.

\begin{lemma}
Let $G$ be a finite $p$-group with dense solitary subgroups. Then $G$ is cyclic. 
\end{lemma}

\begin{proof}
Let $|G|=p^n$, where $n\in\mathbb{N}^*$. If $n\geq 2$, then $G$ contains a subgroup $H$ of order $p^2$. By hypothesis, between $1$ and $H$ there is a solitary subgroup $K$. Then $|K|=p$ and $K$ is the unique subgroup of order $p$ of $G$. By Theorem A, it follows that $G$ is either cyclic or a generalized quaternion $2$-group. Assume that $G\cong Q_{2^n}$ for some $n\geq 3$. Then $|Z(G)|=2$ and $G$ has a subgroup $Q\cong Q_8$. Since all subgroups of order $4$ of $G$ are cyclic, we infer that there is no solitary subgroup $X$ such that $Z(G)<X<Q$, contradicting the hypothesis. Thus $G$ is cyclic, as desired.
\end{proof}

The second one describes the lattice of solitary subgroups of a ZM-group.

\begin{lemma}
We have 
\begin{equation}
{\rm Sol}({\rm ZM}(m,n,r))=\left\{H_{(m_1,n_1,0)}\in L({\rm ZM}(m,n,r))\,:\, m_1|r^{n_1}-1\right\}.\nonumber
\end{equation}
\end{lemma}

\begin{proof} By \cite{2}, two subgroups of ${\rm ZM}(m,n,r)$ are conjugate if and only if they have the same order. This implies that two subgroups of ${\rm ZM}(m,n,r)$ are isomorphic if and only if they are conjugate, and so the solitary subgroups of ${\rm ZM}(m,n,r)$ coincide with the normal ones. The result now follows from Theorem 1 of \cite{14}.
\end{proof}

We are now able to prove our main result.\newpage

\noindent{\bf Proof of Theorem 1.1.} 
Assume that $G$ has dense solitary subgroups. Then all Sylow subgroups of $G$ have this property, and so they are cyclic by Lemma 2.1. If $G$ is not cyclic, then there is a triple of positive integers $(m,n,r)$ such that $G={\rm ZM}(m,n,r)$. Note that $d=o_m(r)\neq 1$. Also, we have
\begin{equation}
{\rm Sol}(G)=\left\{H_{(m_1,n_1,0)}\in L({\rm ZM}(m,n,r))\,:\, m_1|r^{n_1}-1\right\}
\end{equation}by Lemma 2.2.

Assume that $m$ is not prime. Since $G$ is supersolvable, its subgroup $H_{(m,1,0)}=\langle b\rangle$ of index $m$ is not maximal. By hypothesis, there is $H_{(m_1,n_1,0)}\in{\rm Sol}(G)$ with
\begin{equation}
H_{(m,1,0)}<H_{(m_1,n_1,0)}<H_{(1,1,0)}=G.\nonumber
\end{equation}Then $n_1=1$ and $m_1$ is a proper divisor of $m$, a contradiction. Thus $m$ is prime and the equality (1) becomes
\begin{equation}
{\rm Sol}(G)=\left\{H_{(1,n_1,0)}\,:\, n_1\,|\,n\right\}\cup\left\{H_{(m,n_1,0)}\,:\, d\,|\,n_1\right\},
\end{equation}that is ${\rm Sol}(G)$ consists of all subgroups containing $H_{(1,n,0)}=\langle a\rangle$ and of all subgroups contained in $H_{(m,d,0)}=\langle b^d\rangle=Z({\rm ZM}(m,n,r))$. 

Since the interval $(\langle b^d\rangle,\langle b\rangle)$ of $L(G)$ contains no solitary subgroup, we infer that $\langle b^d\rangle$ is a maximal subgroup of $\langle b\rangle$, implying that $d$ is prime. 

If $d=n$, we are done. Suppose that $d\neq n$. Then either $\langle b^{\frac{n}{d}}\rangle$ is a maximal subgroup of $\langle b\rangle$ or the interval $(\langle b^{\frac{n}{d}}\rangle,\langle b\rangle)$ of $L(G)$ contains a solitary subgroup. In the first case, we get $$n=d^2 \mbox{ or } n=dp, \mbox{ where } p \mbox{ is a prime with } p\neq d.$$In the second case, let $H_{(m,n_1,0)}\in{\rm Sol}(G)$ such that
\begin{equation}
\langle b^{\frac{n}{d}}\rangle<H_{(m,n_1,0)}<\langle b\rangle.\nonumber
\end{equation}Then 
\begin{equation}
d\,|\,\frac{n}{n_1}\,|\,n.\nonumber
\end{equation}Since $d|n_1$, it follows that the set $\{d'\,:\,d\,|\,d'\,|\,n\}$ must contain divisors of $\frac{n}{d}$\,, and therefore $d\,|\,\frac{n}{d}$\,, that is $d^2\,|\,n$. Let $n=d^{\alpha}n'$ with $\alpha\geq 2$ and $d\nmid n'$. 

If $n'=1$, we are done. Suppose that $n'\neq 1$. Since there is no divisor $d'$ of $\frac{n}{d}$ such that $d^{\alpha}\,|\,d'\,|\,n$, we infer that $d^{\alpha}$ must be a maximal divisor of $n$. Thus $n'$ is prime.

The converse is obvious, completing the proof.
$\square$
\newpage

\bigskip{\bf Funding.} The author did not receive support from any organization for the submitted work.

\bigskip{\bf Conflicts of interests.} The author declares that he has no conflict of interest.

\bigskip{\bf Data availability statement.} My manuscript has no associated data.

\vspace*{5ex}\small

\hfill
\begin{minipage}[t]{5cm}
Marius T\u arn\u auceanu \\
Faculty of  Mathematics \\
``Al.I. Cuza'' University \\
Ia\c si, Romania \\
e-mail: {\tt tarnauc@uaic.ro}
\end{minipage}

\end{document}